\documentclass{birkjour}

\usepackage{amsmath, amsfonts, amssymb, latexsym}
\usepackage{tabularx, longtable, makecell, arydshln, float}
\allowdisplaybreaks

\newtheorem{defn}{Definition}

\newtheorem{thm}{Theorem}

\newtheorem{cor}{Corollary}

\def\beq{\begin{equation}}
\def\eeq{\end{equation}}
\def\beqs{\begin{equation*}}
\def\eeqs{\end{equation*}}

\newcommand{\R}{\mathbb{R}}
\newcommand{\cH}{\mathcal{H}}

\newcommand{\la}{\lambda}

\newcommand{\h}{\mathcal{H}}

\begin{document}

\title{On fixed point approach to equilibrium problem}

\author{Le Dung Muu}
\address{Institute of Mathematics and Applied Sciences\\
	Thang Long University\\
	Hanoi\\
	Vietnam}
\email{ldmuu@math.ac.vn}

\author{Xuan Thanh Le}
\address{Institute of Mathematics\\
	Vietnam Academy of Science and Technology\\
	Hanoi\\
	Vietnam}
\email{lxthanh@math.ac.vn}

\begin{abstract}
The equilibrium problem defined by the Nikaid\^o-Isoda-Fan inequality contains a number of problems such as optimization, variational inequality, Kakutani fixed point, Nash equilibria, and others as special cases. This paper presents a picture for the relationship between the fixed points of the Moreau proximal mapping and the solutions of the equilibrium problem that satisfies some kinds of monotonicity and  Lipschitz-type condition.
\end{abstract}



\maketitle

%%%%%%%%%%%%%%%%%%%%%%%%%%%%%%%%%%%%%%%%

\section{Introduction}
In this paper we are concerned with the equilibrium problem stated as
\beq
\text{Find } x^* \in C \text{ such that } f(x^*,y) \geq 0 \text{ for all } y \in C, \tag{EP}
\eeq
in which $C$ is a nonempty closed convex subset in a Hilbert space $\cH$ endowed with an inner product $\langle \cdot, \cdot \rangle$ and the induced norm $\| \cdot \|$, and $f: \cH \times \cH \to \R \cup \{+\infty\}$ is a bifunction
such that $f(x, y) < +\infty$ for every $x, y \in C$. The inequality in Problem (EP) was first used in \cite{NI1955} for convex noncooperative game theory. The first result on solution existence of (EP) is due to K. Fan \cite{F1972}, where this problem was called a minimax inequality. The name equilibria was first used in \cite{MO1992}. After the appearance of the paper by Blum and Oettli \cite{BO1994}, the problem (EP) has attracted much attention of many authors and a lot of algorithms have been developed for solving the problem where the bifunction $f$ have monotonic properties.  These algorithms are based upon different methods such as penalty and gap functions \cite{B1P2012, B2P2015, B5P2016, K2D2003,  K4D2003, K3D2003, M2003, MO1992}, regularization \cite{AA2020, HM2011, LA2013, MQ2009, MQ2010}, extragradient methods \cite{B6P2017, HSM2020, K5S2005, QAM2012, SS2011, SS2015, SSA2011, TH2017, VM2019, VSN2012, YM2020}, splitting technique \cite{AH2017,DM2016,ML2018}. A comprehensive reference-list on algorithms for the equilibrium problem can be found in the interesting monograph \cite{BCPP2019}.

An interesting of this problem is that, despite its simple formulation, it contains many problems such as optimization, reverse optimization, variational inequality, minimax, saddle point, Kakutani fixed point, Nash equilibrium problems, and some others as special cases (see the interesting monographs \cite{BCPP2019, KR2018} and the papers \cite{BO1994, MO1992}).

In what follows we  always suppose that $\phi: C \times C \to \R$ such that  $\phi(x,\cdot)$ is convex for any $x \in C$, and $\varphi: C \to \R$ is a convex function on $C$. For continuity (resp. lower and upper continuity) of $\phi$ and $\varphi$ we means the continuity (reps. lower and upper continuity) with respect to the set $C$.  Then we consider Problem (EP) with
$f(x,y) :=\phi(x,y) + \varphi(y) - \varphi(x)$. In this case Problem (EP) becomes a mixed equilibrium problem of the form
\beq
\text{Find } x^* \in C \text{ such that } f(x^*,y):= \phi(x^*,y) + \varphi(y) - \varphi(x^*) \geq 0 \text{ for all } y \in C. \tag{MEP}
\eeq
By considering this mixed form one can employ special structures of each $\phi$ and $\varphi$ in subgradient  splitting algorithms, where the bifunction $f(x,y)$ can be expressed by the sum of two bifunctions $f_1(x,y) + f_2(x,y)$ and the iterates are defined by taking the proximal mappings of each $f_1$ and $f_2$ separately, see \cite{AH2017,AA2020,DM2016,MQ2010,ML2018}.

The first fixed point approach to equilibrium problem (EP) was first developed in 1972 by K. Fan in \cite{F1972}. There, by using the KKM lemma, it has been proved that if $C$ is compact and $f(x,\cdot)$ is quasiconvex on $C$, then Problem (EP) admits a solution under a certain continuity property of $f$. Note that in this {result of K. Fan, it does not require any monotonicity of the bifunction $f$.

A direct proof using the Kakutani fixed point theorem for the solution existence of Problem (EP) is based upon the mapping $K$ defined by taking, for each $x\in C$,
\beq
K(x):= \text{argmin} \{ f(x, y): y \in C\}. \tag{P1}
\eeq
Clearly, if $f(x,x) = 0$ for every $x\in C$, then $x^*$ is a solution to (EP) if and only if it is a fixed point of $K$, i.e., $x^* \in K(x^*)$. Thus, if $C$ is convex, compact, $f(x,\cdot)$ is convex on $C$ and $K$ is upper semicontinuous on $C$, then by the well known Kakutani fixed point theorem, the mapping $K$ has a fixed point.  It can be noticed that the mapping $K$ is set-valued in general.

In order to avoid multivalues of $K$, an auxiliary principle has been used by defining the proximal mapping
\beq
B_{\lambda}(x) := \text{argmin} \left\{\lambda f(x,y) + \frac{1}{2} \langle y-x, G(y-x)\rangle \right\}, \tag{P2}
\eeq
where $\lambda > 0$ and $G$ is a self-adjoint positive linear bounded operator from $\cH$ into itself. In the sequel, for simplicity of the presentation, we always suppose that $G$ is the identity operator. It is well known \cite{BV2004} that if $f(x,\cdot)$ is convex and subdifferentiable  on $C$,  Problem (P2) is uniquely solvable even for the case $C$ is not compact. Moreover, a point $x^* \in C$ is a solution of Problem (EP) if and only if $x^*$ is a fixed point of $ B_{\lambda}$ for any $\lambda > 0$. So the solution existence of (EP) can be proved by using the Brouwer fixed point theorem whenever $C$ is compact and $B_\la$ is continuous on $C$.

These results suggest  that one can apply the existing algorithms such as ones based on the Scaft pivoting method \cite{KM1975} for computing a fixed point of the mapping $B_{\lambda}$, thereby solving the equilibrium problem (EP). However, the computational results \cite{KM1975,TTM1978} show that  the pivoting methods are efficient only for problems with moderate size. Since in the fixed point theory, iterative methods for computing a fixed point have been successfully applied to contractive, generalized contractive, and nonexpansive  mappings, a natural question arises that under which conditions, the mapping $B_{\lambda}$ possesses certain contraction or generalized nonexpansiveness properties?

This is a survey paper, but it also contains some new results on a fixed point approach to equilibrium problem (MEP). Namely, first we outline results on quasicontraction and contraction of the Moreau proximal mapping when the bifunction $f$ is strongly monotone and satisfies a certain Lipschitz-type condition. Next, in the case $f$ is not necessarily strongly monotone, but monotone, we present some results on approximate nonexpansiveness of the proximal mapping for monotone equilibrium problems satisfying a certain strongly Lipschitz-type condition. Finally, we show a result on quasinonexpansiveness of a composed proximal mapping defined by the equilibrium problem. This relationship allows that the equilibrium problem can be solved by the existing methods in the fixed point theory (see e.g. \cite{C2013, GD1997, I1974, M1998, TT2007, TX1993} and the references theirein).

The paper is organized as follows. The next section contains preliminaries on the equilibrium problem under consideration and on generalized contractions in real Hilbert spaces. In Section \ref{ContractionGNP} we present some results on contraction, quasicontraction, nonexpansiveness, and approximate nonexpansiveness of the Moreau  proximal mapping defined for the equilibrium problem. We close the paper by some conclusions in Section \ref{ConclusionSection}.

%%%%%%%%%%%%%%%%%%%%%%%%%%%%%%%%%%%%%%%%%%%%%%%%%%%%%%%%
	
\section{Preliminaries}

The following definitions for a bifunction is commonly used in the literature, see e.g. \cite{BCPP2019}.

\begin{defn}\label{Sect2Defn1}
A bifunction  $f:C\times C \to \R$ is said to be

(i) strongly monotone with modulus $\gamma >0$ (shortly $\gamma$-strongly monotone) on $S \subseteq C$ if
\beqs
f(x,y) + f(y,x) \leq -\gamma \|x-y\|^2 \ \forall x,y \in S;
\eeqs

(ii) monotone on $S \subseteq C$ if
\beqs
f(x,y) + f(y,x) \leq 0 \ \forall x, y \in S;
\eeqs

(iii) strongly pseudomonotone on $S \subseteq C$ with modulus $\gamma > 0$ (shortly $\gamma$- strongly pseudomonotone) if
\beqs
f(x,y) \geq 0 \Rightarrow f(y,x) \leq -\gamma \|x-y\|^2 \ \forall x, y \in S;
\eeqs

(iv) pseudomonotone on $S \subseteq C$ if
\beqs
f(x,y) \geq 0 \Rightarrow f(y,x) \leq 0 \  \forall x, y \in S.
\eeqs
\end{defn}

The notions on monotonicity properties of a bifunction are generalized ones for operators.  In fact, it is easy to see that when $f(x,y) := \langle F(x), y-x\rangle +\varphi (y) - \varphi (x)$, then $f$ is $\gamma$-strongly monotone (resp. monotone, $\gamma$-strongly pseudomonotone, pseudomonotone) if and only if $F$ is $\gamma$-strongly monotone (resp. monotone, $\gamma$-strongly pseudomonotone, pseudomonotone).
The following Lipschitz-type conditions has been introduced in \cite{M2003} and  commonly used for Problem (EP).

\begin{defn}\label{Sect2Defn2}
A bifunction  $f: C\times C \to \R$ is said to be of {\it Lipschitz-type} on $S \subseteq C$ if there exists constants $L_1, L_2 >0$ such that
\beqs
f(u,v) + f(v,w) \ge f(u, w) - L_1 \|u-v\|^2 - L_2\|v-w\|^2 \ \forall u, v, w \in S.
\eeqs
\end{defn}

By taking $u=v=w$ we see that $f(u, u) \geq  0$, so if, in addition, $f$ is pseudomonotone, then $f(u,u) = 0$.

The following concepts are well-known in the fixed point theory (see e.g. \cite{AOS2009}).

\begin{defn}\label{Sect2Defn3}
Let
 $T: \h \to C$.

 (i) $T$ is said to be {\it contractive} on $C$ if there exists $0 < \rho < 1$ such that
\beqs
\|T(x) - T(y)\| \leq \rho \|x-y\| \ \forall x, y \in C.
\eeqs
If $T$ satisfies this condition with $\rho = 1$, then it is said to be {\it nonexpansive}. It is said to be {\it quasicontractive} on $C$ if
\beqs
\|T(x) - T(y) \| \leq \rho \|x-y\| \ \forall  x \in {\rm Fix}(T), y \in C,
\eeqs
where ${\rm Fix}(T)$ stands for the set of fixed points of $T$. If this condition holds for $\rho = 1$, then $T$ is said to be {\it quasinonexpansive}.

(ii) $T$ is said to be {\it firmly nonexpansive on $C$} if
\beqs
\|T(x) - T(y) \|^2 \leq  \|x-y\|^2- \|(I-T)(x) - (I-T)(y)\|^2 \ \forall  x, y \in C.
\eeqs

(iii) $T$ is said to be {\it $\rho$-strongly converse monotone or $\rho$-cocoercive on $C$} with $\rho > 0$, if
\beqs
\langle T(x) -T(y), x-y\rangle \geq \rho \|T(x) -T(y)\|^2 \ \forall  x, y \in C.
\eeqs
\end{defn}

%%%%%%%%%%%%%%%%%%%%%%%%%%%%%%%%%%%%%%%%%%%%%%%%%%%%%%%%

\section{Contraction and generalized nonexpansive properties of the proximal mapping}\label{ContractionGNP}

Let $g: C \to \R$ be a convex function and $\lambda > 0$. The proximal mapping $P_{\lambda}$ with respect to $C, g, \lambda$ (shortly proximal mapping) is defined as follows (see e.g. \cite{RW1998}):
\beqs
P_{\lambda}(x) := \text{argmin}\left\{ \lambda g(y) + \frac{1}{2}\|y-x\|^2 : y \in C\right\}.
\eeqs
For the bifunction $f$ where $f(x,\cdot)$ is convex and finite on $C$, the proximal mapping $B_{\lambda}$ is defined by taking
\beqs
B_{\lambda}(x) := \text{argmin} \left\{ \lambda f(x,y) + \frac{1}{2} \|y-x\|^2 : y \in C\right\}
\eeqs
for each $x \in C$. Note that when either $f(x,\cdot)$ is continuous on $C$ or $C$ has an interior point, and $f(x,x) = 0$, then it is well known from \cite{BCPP2019} that $x^*$ is a fixed point of $B_{\lambda}$ if and only if it is a solution to Problem (EP).

The following theorem says that when $f$ is strongly monotone and satisfies the Lipschitz-type on $C$, then one can choose a regularization parameter such that the proximal mapping is quasicontractive on $C$. For applying optimality condition for the problem defining the proximal mapping, we always assume that either $C$ has an interior point or, for any $x \in C$, $\phi(x,\cdot)$ is continuous with respect to $C$ at a point of $C$.

\begin{thm}\label{T1}
Suppose that $f$ is strongly monotone on $C$ with modulus $\tau$ and satisfies the Lipschitz-type condition with constants $L_1, L_2$ satisfying $L_1 + L_2 > \tau$. Then, the proximal mapping $B_{\lambda}$ is quasicontractive on $C$, namely
\beqs
\|B_{\lambda} (x) - x^*\| \le \sqrt{\alpha} \|B_{\lambda}(x)- x^*\| \ \forall x \in C, x^* \in {\rm Fix}(B_{\lambda}),
\eeqs
whenever $\lambda \in (0, \frac{1}{2L_2})$, where $\alpha := 1 -2\lambda (\tau-L_1) > 0$.
\end{thm}
\begin{proof}
The following proof borrows some techniques from the one in \cite{QAM2012}. For simplicity of notation we let
\beqs
f_x(y): = \lambda f(x,y) + \frac{1}{2} \|y-x\|^2.
\eeqs
Since $\lambda > 0$ and $f(x, \cdot)$ is convex on $C$ by assumption, $f_x$ is strongly convex with modulus 1. As defined, $B_{\lambda}(x)$ is a minimizer of $f_x(\cdot)$ over the closed convex set $C$. Therefore, we have
\beqs
f_x(B_{\lambda}(x)) + \frac{1}{2} \|B_{\lambda}(x) - x^*\|^2 \le f_x(x^*),
\eeqs
that is
\beqs
\lambda f(x, B_{\lambda}(x)) + \frac{1}{2} \|B_{\lambda}(x) - x\|^2 + \frac{1}{2} \|B_{\lambda}(x) - x^*\|^2 \le \lambda f(x, x^*) + \frac{1}{2} \|x^* - x\|^2,
\eeqs
or equivalently
\beq\label{1}
\|B_{\lambda}(x) - x^*\|^2 \leq 2\lambda \left( f(x,x^*) - f(x, B_{\lambda}(x)) \right) + \|x-x^*\|^2 - \|B_{\lambda}(x) - x\|^2.
\eeq
Since $f$ is strongly monotone on $C$ with modulus $\tau$, it follows from (\ref{1}) that
\begin{align}
	& \ \|B_{\lambda}(x) - x^*\|^2  \notag\\
	\le & \ 2\lambda (-\tau \|x - x^*\|^2 - f(x^*, x) - f(x, B_{\lambda}(x))) + \|x-x^*\|^2 -\|B_{\lambda}(x) - x\|^2 \notag\\
	\le & \ (1 - 2\lambda\tau) \|x - x^*\|^2 - 2\lambda \left(f(x^*, x) + f(x, B_{\lambda}(x)) \right) - \|B_{\lambda}(x) - x\|^2.\label{ev2}
\end{align}
Since $f$ satisfies Lipschitz-type condition, we have
\beqs
f(x^*, x) + f(x, B_{\lambda}(x)) \ge f(x^*, B_{\lambda}(x)) - L_1 \|x^* - x\|^2 - L_2 \|x - B_{\lambda}(x) \|^2.
\eeqs
Therefore, it follows from (\ref{ev2}) that
\begin{align*}
	& \ \|B_{\lambda}(x) - x^*\|^2\\
	\le & \  (1 - 2\lambda(\tau - L_1)) \|x - x^*\|^2 - (1 - 2\lambda L_2) \|x - B_{\lambda}(x)\|^2 - 2\lambda f(x^*, B_{\lambda}(x))\\
	\le & \ (1 - 2\lambda(\tau - L_1)) \|x - x^*\|^2.
\end{align*}
The last inequality is due to the fact that $f(x^*, B_{\lambda}(x)) \ge 0$ and the assumption that $0 \le \lambda \le \frac{1}{2L_2}$. Since $\tau < L_1 + L_2$, we see that if $0< \lambda < \frac{1}{2L_2}$, then $1-2\lambda(\tau-L_1) < 1$ and hence $B_{\lambda}$ is quasicontractive.
\end{proof}

Theorem \ref{T1} allows that the contraction iterative method can be used for solving equilibrium problem (EP). A question arises here: under which condition, the proximal mapping is contractive? The following theorem in \cite{H2017} gives an answer for this question.

For the statement of the theorem, first we recall from \cite{H2017} that a bifunction $f: C\times C \to \R$ is said to be {\it strongly Lipschitz-type} on $C$ if there exist $\alpha_i: C \times C \to C$, $\beta_i : C\to C$, $K_i, L_i > 0$ (with $i=1,\ldots,p$) such that
\beqs
f(x,y) + f(y,z) \ge f(x,z) + \sum_{j=1}^p \langle \alpha_i(x,y) , \beta_i(y-z) \rangle \   \forall x,y, z \in C,
\eeqs
where
\begin{align*}
	\|\beta_i(x) - \beta_i(y)\| &\leq K_i\|x-y\| &&\forall x, y \in C, i = 1, \ldots, p,\\
	\|\alpha_i(x,y)\| &\leq L_i\|x-y\| &&\forall x, y \in C, i = 1, \ldots, p,\\
	\alpha_i(x,y) + \alpha_i(y,x) &= 0 &&\forall x, y \in C, i = 1, \ldots, p.
\end{align*}
As also remarked in \cite{H2017}, the following facts are not hard to see.

(i) If  $f$ is strongly Lipschitz-type  on $C$, then it is Lipschitz-type on $C$ with  both constants $\frac{1}{2} \sum_{i=1}^p K_i L_i$.

(ii) If $f(x,y) = \langle F(x), y-x \rangle + \varphi(y) - \varphi(x)$, then $f$ is strongly Lipschitz-type if and only if $F$ is Lipschitz.

\begin{thm}\label{T2}
Let $C$ a be nonempty closed convex set, $f: C\times C\to \R$. Suppose that $f(x,\cdot)$ is lower semicontinuous, convex, $\gamma$-strongly monotone, and strongly Lipschitz-type on $C$. Then the proximal mapping $B_{\lambda}$ is contractive on $C$ whenever $\lambda \in (0, \frac{2\gamma}{M})$ with $M = \sum_{i=1}^p K_i L_i.$
\end{thm}

In the case of mixed variational inequality, when $f(x,y) := \langle F(x), y-x \rangle + \varphi(y) - \varphi(x)$ with $F$ being Lipschitz continuous and strongly monotone, the proximal mapping is contractive (see. e.g. \cite{AMNS2005}). This result also follows from the above theorem due to the fact that $f(x,y) := \langle F(x), y-x\rangle + \varphi (y) - \varphi (x)$  is strongly Lipschitz-type on $C$ whenever $F$ is  Lipschitz on $C$. In case that $F$ is co-coercive (strongly inverse monotone), the proximal mapping is nonexpansive on $C$ as stated in the following theorem.

\begin{thm}\label{T3}
Suppose that $f(x,y) := \langle F(x), y-x \rangle + \varphi(y) - \varphi(x)$ with $F$ being $\delta$-co-coercive (or strongly inverse monotone) on $C$. Then, whenever $0 < \lambda \leq \frac{1}{2\delta}$, the proximal mapping $B_{\lambda}$ is nonexpansive on $C$.
\end{thm}
\begin{proof}
From the definition of $B_{\lambda}$, by using the optimization condition, it is easy to see that
\beq\label{4}
\|B_{\lambda}(x) - B_{\lambda}(y)\|^2 \leq \|x-y- \lambda ( F(x) - F(y) )\|^2 \ \forall x,y \in C.
\eeq
Since $F$ is co-coercive on $C$ with modulus $\delta$, we have
\beqs
\langle x-y, F(x) - F(y)\rangle \geq \delta\|F(x) - F(y)\|^2,
\eeqs
which implies
\beqs
\| x-y - \lambda (F(x) - F(y))\| \leq \|x-y\|.
\eeqs
Thus by (\ref{4}), we have
\beqs
\|B_{\lambda}(x) - B_{\lambda}(y) \| \leq \|x-y\|.
\eeqs
\end{proof}

A question now may arise that is the proximal mapping $B_{\lambda}$ nonexpansive when $f$ is monotone? The following simple example gives a negative answer.

Let us consider the linear variational inequality
\beq\label{VI}
\text{Find} \ x \in \R^2 \text{ such that } f(x, y) := \langle Ax, y-x \rangle \geq 0 \text{ for all } y \in \R^2,\tag{VI}
\eeq
where
\beqs
A = \begin{bmatrix} 0 & 1\\ -1 & 0 \end{bmatrix}.
\eeqs
For all $x, y \in \R^2$ we have
	\begin{align*}
		f(x, y) + f(y, x) &= \langle A(x-y), y - x \rangle\\
		&= (x_2 - y_2)(y_1 - x_1) - (x_1 - y_1)(y_2 - x_2)\\
		&= 0,
	\end{align*}
	therefore $f$ is monotone on $\R^2$. It is easy to see that $x^* =  (0, 0)^t$ is a solution to the variational inequality (VI), since $f(x^*, y) = 0$ for all $y \in \R^2$. Furthermore, $x^*$ is the unique solution to (VI). Indeed, if $\overline{x} = (\overline{x}_1, \overline{x}_2)^t$ is a solution to (VI), then
	\beqs
	\langle A\overline{x}, y - \overline{x} \rangle \ge 0 \quad \forall y \in \R^2.
	\eeqs
	By taking $y = \overline{y} = (\overline{x}_1 - \overline{x}_2, \overline{x}_1 + \overline{x}_2)^t$, we have
	\beqs
	0 \le  \langle A\overline{x}, \overline{y} - \overline{x} \rangle = \left\langle \begin{bmatrix} \overline{x}_2 \\ -\overline{x}_1 \end{bmatrix}, \begin{bmatrix} -\overline{x}_2 \\ \overline{x}_1 \end{bmatrix} \right\rangle = -(\overline{x}_1^2 + \overline{x}_2^2) \le 0,
	\eeqs
	which implies $\overline{x} = (0, 0)^t = x^*$. Now we see that
	\begin{align*}
	& \lambda \langle Ax,y-x\rangle + \frac{1}{2}\|y-x\|^2 \\
	= & \lambda \left\langle \begin{bmatrix} x_2 \\ -x_1 \end{bmatrix}, \begin{bmatrix} y_1 - x_1 \\ y_2 - x_2 \end{bmatrix} \right\rangle + \frac{1}{2} \left( (y_1 - x_1)^2 + (y_2 - x_2)^2 \right)\\
	= & \lambda (x_2 y_1 - x_1 y_2) + \frac{1}{2} \left( (y_1 - x_1)^2 + (y_2 - x_2)^2 \right)	\\
	= & \frac{1}{2} \left(y_1^2 - 2(x_1 - \lambda x_2)y_1 + y_2^2 - 2(x_2 + \lambda x_1) y_2  + x_1^2 + x_2^2 \right)\\
	= & \frac{1}{2} \left( (y_1 - x_1 + \lambda x_2)^2 + (y_2 - x_2 - \lambda x_1)^2 - \lambda^2 (x_1^2 + x_2^2) \right)
	\end{align*}
	which attains its minimum at $(y_1, y_2) = (x_1 - \lambda x_2, x_2 + \lambda x_1)$. We obtain the following explicit formula for the proximal mapping of (VI):
	\beqs
	B_{\lambda}(x) = \text{argmin} \{\lambda \langle Ax,y-x\rangle + \frac{1}{2}\|y-x\|^2 \mid y \in \R^2\} = \begin{bmatrix} x_1 - \lambda x_2 \\ x_2 + \lambda x_1 \end{bmatrix}.
	\eeqs
	Therefore, for any $\lambda > 0$ we have
	\beqs
	\|B_{\lambda}(x) - B_{\lambda}(x^*)\| = \left\| \begin{bmatrix} x_1 - \lambda x_2 \\ x_2 + \lambda x_1 \end{bmatrix} - \begin{bmatrix} 0 \\ 0 \end{bmatrix}\right\| = \sqrt{1 + \lambda^2} \sqrt{x_1^2 + x_2^2} > \|x - x^*\|,
	\eeqs
which proves that $B_{\lambda}$ is not nonexpansive for any $\lambda > 0$.

So in general the proximal mapping may not be nonexpansive even for the variational inequality when $f(x,y) = \langle F(x), y-x \rangle$ with $F$ being Lipschitz and monotone. However, it is well known from \cite{CH2005, TT2007} that if $f$ is monotone, $f(x, \cdot)$ is convex, lower semicontinuous, and $f(\cdot,y)$ is hemicontinuous, then the regularization proximal mapping $R_{\lambda}$ is defined everywhere, single valued, and firmly nonexpansive for any $\lambda > 0$. Here, for each $x \in C$, $R_{\lambda}(x)$ is defined as the unique solution of the strongly monotone equilibrium problem
\beqs
\text{Find } z \in C \text{ such that } f(z,y) + \frac{1}{2\lambda} \langle y-z, z-x \rangle \geq 0 \text{ for all } y\in C.
\eeqs
Moreover, the solution set of (EP) coincides with the fixed point set of the proximal mapping $R_{\lambda}$.

The main difference between the proximal mapping and the regularization proximal mapping is that the former is defined as the unique solution of a strongly convex program, while the latter is defined by the unique solution of a strongly monotone equilibrium problem.

We adopt the following definition.

\begin{defn}
For given $\epsilon > 0$ and $\lambda > 0$, the proximal mapping $B_{\lambda}$ is said to be {\it $\epsilon$-nonexpansive on $C$} if
\beqs
\|B_{\lambda}(x) - B_{\lambda}(y)\|^2 \leq (1 + \epsilon) \|x - y\| ^2 \  \forall x, y \in C.
\eeqs
\end{defn}

The following theorem says that for monotone equilibrium problem, the proximal mapping is $\epsilon$-nonexpansive.

\begin{thm}\label{T4}
Suppose that the bifunction  $\phi$ is monotone and satisfies the strongly Lipschitz-type condition on $C$. Then for any $\epsilon > 0$, there exists $\lambda > 0$ such that the proximal mapping $B_{\lambda}$ for Problem (MEP) is $\epsilon$-nonexpansive.
\end{thm}
\begin{proof}
As before we see that if $\phi$ is monotone, strongly Lipschitz-type, then so is $f(x,y):= \phi(x,y) + \varphi(y)-\varphi (x)$ for any function $\varphi: C \to \R$. It is well known (see. e.g. \cite{M1965, RW1998}) that
\beqs
\langle B_{\lambda}(x) - x, B_{\lambda}(x) - z \rangle \leq \lambda [f(x, z) - f(x, B_{\lambda}(x) ] \ \forall x, z\in C.
\eeqs
Applying this inequality with $z:= B_{\lambda}(y)$ we obtain
\beqs
\langle B_{\lambda}(x)- x, B_{\lambda}(x)- B_{\lambda}(y) \rangle \leq \lambda [f(x, B_{\lambda}(y)) - f(x, B_{\lambda}(x))] \ \forall x, y\in C.
\eeqs
Similarly with $B_{\lambda}(y)$, we have
\beqs
\langle B_{\lambda}(y) - y, B_{\lambda}(y)- B_{\lambda}(x) \rangle \leq \lambda[f(y, B_{\lambda}(x)) - f(y, B_{\lambda}(y))] \ \forall x, y \in C.
\eeqs
Adding the two obtained inequalities we get
\begin{align*}
&\langle B_{\lambda}(x) - B_{\lambda}(y) + y - x , B_{\lambda}(x)- B_{\lambda}(y) \rangle \\
\leq \ &\lambda [f(x, B_{\lambda}(y)) - f(x, B_{\lambda}(x) + f(y, B_{\lambda}(x)) - f(y, B_{\lambda}(y))].
\end{align*}
By simple arrangements we obtain
\begin{align*}
&\|B_{\lambda}(x) - B_{\lambda}(y)\|^2 \\
+ \ &2\lambda \Big[  f(x, B_{\lambda}(y)) - f(x, B_{\lambda}(x) + f(y, B_{\lambda}(x)) - f(y, B_{\lambda}(y)) \Big]\\
\le \ &\| x-y\|^2
\end{align*}
Now using the strongly Lipschit-type condition, by the same argument as in the proof of Theorem 3.7 in \cite{H2017} we arrive at the following inequality
\beqs
\|B_{\lambda}(x) - B_{\lambda}(y)\|^2 \leq (1+\lambda^2 M) \|x - y\|^2,
\eeqs
where $M = \sum_{j=1}^p K_i L_i$, with $K_j$ and $L_j$ being the Lipschitz constants defined in the strongly Litschitz-type. Hence, with $0 < \lambda^2 < \frac{\epsilon}{M}$, we obtain $\|B_{\lambda}(x) - B_{\lambda}(y)\|^2 
\leq (1 +\epsilon )\|x-y\|^2$ for every $x, y \in C$.
\end{proof}

\begin{cor}\label{CH2005} Consider the mixed variational inequality
\beqs
\text{Find } x^* \in C \text{ such that } \langle F(x^*), y-x^*\rangle + \varphi(y) - \varphi(x^*) \geq 0 \text{ for all } y\in C. \tag{MVI}
\eeqs
Suppose that $F$ is monotone and  Lipschitz on $C$. Then the proximal mapping
$ B_{\lambda}$ defined by the bifunction $ \langle F(x), y-x\rangle + \varphi(y) -\varphi (x)$ is $\epsilon$-nonexpansive for any $\epsilon > 0$.
\end{cor}
\begin{proof}
Since $F$ is monotone and Lipschitz on $C$, the bifunction $f(x,y) := \langle F(x), y-x\rangle + \varphi(y) -\varphi(x)$ is strongly Lipschitz and monotone. Thus the corollary follows directly from Theorem \ref{T4}. However one can prove this result simply as follows.

From the definition of $B_{\lambda}$, by using the optimality condition for the problem defining $B_{\lambda}$ we can show
\beq\label{ct3}
\| B_{\lambda}(x) -B_{\lambda}(y)\| \leq \|x-y-\la( F(x) - F(y) )\| \  \forall x, y \in C.
\eeq
Then, from
\beqs
\|x-y-\lambda( F(x) - F(y) )\|^2 = \|x-y\|^2 -2\lambda\langle x-y, F(x) - F(y)\rangle +\lambda^2\|F(x) -F(y)\|^2
\eeqs
by (\ref{ct3}) and monotonicity of $F$ we can write
\beqs
\|B_{\lambda}(x) -B_{\lambda}(y)\|^2 \leq \|x-y\|^2  + \lambda^2\|F(x)-F(y)\|^2,
\eeqs
from which, by Lipschitz continuity of $F$, it follows that
\beqs
\|B_{\lambda}(x) - B_{\lambda}(y)\|^2 \leq (1+L^2\la^2)\|x-y\|^2.
\eeqs
Hence the mapping $B_{\lambda}$ is $\epsilon$-nonexpansive on $C$ whenever $\lambda^2 L^2 \leq \epsilon$.	
\end{proof}

Now a natural question may arise: how to modify the proximal mapping for monotone equilibrium problems such that it has a generalized nonexpansiveness property? In order to answer this question, let us define the mapping $T_{\lambda}$ from $C$ to itself by taking, for every $x\in C$,
\beqs
T_{\lambda}(x):= \text{argmin}\Big\{\lambda f(B_{\lambda}(x), y) + \displaystyle\frac{1}{2}\Vert y-x\Vert^2: y\in C \Big\}
\eeqs
where $\lambda$ is a fixed positive number.

\begin{thm}\label{T5} {\rm (\cite{AM2018})}.
Let $f: C\times C \to \R$ be a bifunction such that $f(x, \cdot)$ is subdifferentiable, pseudomonotone and Lipschitz-type on $C$. Suppose that the following conditions are satisfied:

%\indent (A0) Either $C$ has an interior point or $f(x,.)$ is continuous with respectively to $C$ at a point in $C$;\\

(A1) $f$ is jointly weakly continuous on $C \times C$ in the sense that, if $x, y\in C$ and $\{x_n\}, \{y_n\}\subset C$ converge weakly to $x$ and $y$, respectively, then $f(x_n, y_n) \to f(x, y)$ as $n \to \infty$.

(A2) The solution set of Problem (EP) is nonempty.\\
Then the mapping $T_{\lambda}$ is quasi-nonexpansive on $C$ if $0<\lambda< \min\left\{\displaystyle\frac{1}{2L_1}, \displaystyle\frac{1}{2L_2}\right\}$. In addition, it is demiclosed at zero, in the sense  that for every sequence $\{x_n\}$ contained in $C$ weakly converging to $x$ and $\|T(x_n) - x_n\| \to 0$,  then $x\in Fix(T)$.
\end{thm}

By this theorem, the algorithms for finding a fixed point of quasi-nonexpansive mappings (see e.g. \cite{GD1997, I1974}) can be used for solving pseudomonotone equilibrium problems.

A disadvantage of the composite  proximal mapping $T_{\lambda}$ is that for evaluating it at a point, it requires solving two strongly convex programming problems.
An open question is that how to define a nonexpansive or $\epsilon$-nonexpansive mapping with any $\epsilon >0$ for pseudomotone equilibrium problems, which requires solving only one strongly convex programs?

As we have seen from the definition of the proximal mapping that when applying the iterative fixed point methods for solving mixed equilibrium problem (MEP),  at an iterative  point $x^k \in C$, we have to solve a strongly convex program of the form
\beqs
\min\left\{ f(x^k, y) : = \phi(x^k,y) + \varphi(y) - \varphi (x^k) + \frac{1}{2\lambda}\|y-x^k\|^2: y \in C\right\}. \eqno(P_k)
\eeqs
This problem can be solved by efficient algorithms of convex programming (see \cite{BV2004}).

%%%%%%%%%%%%%%%%%%%%%%%%%%%%%%%%%%%%%%%%%%%%%%%%%%%%%%%%

\section{Conclusions}\label{ConclusionSection}

The mixed equilibrium problem (MEP) and the regularized Moreau proximal mapping $B_{\lambda}$ defined for it are equivalent in the following senses:

(i) The solution set of (MEP) coincides with the fixed point set of $B_{\lambda}$ for any $\lambda > 0$.

In addition:

(ii) If (MEP) is strongly monotone and satisfies the Lipschitz-type condition, then one can choose $\la$ such that $B_{\lambda}$ is quasicontractive.

(iii) If (MEP) is strongly monotone and satisfies the strongly Lipschitz-type condition, then one can choose $\la$ such that $B_{\lambda}$ is contractive.

(iv) If (MEP) is monotone and satisfies the strongly Lipschitz-type, then one can choose $\la$ such that $B_{\lambda}$ is $\epsilon$-nonexpansive for any $\epsilon > 0$.

(v) If (MEP) is pseudomonotne and satisfies the Lipschitz-type, then the composite proximal mapping is quasinonexpansive, and its fixed point-set coincides the solution-set of Problem (MEP).

Applications to mixed variational inequality problems with Lipschitz cost operator have been presented.

The following question seems to be interesting: How to extend these results for the equilibrium problem when the bifunction involved is quasiconvex with respect to its second variable?

\section*{Acknowledgements}

This research is funded by Vietnam National Foundation for Science and Technology Development (NAFOSTED) under grant number 101.01-2020.06.

%%%%%%%%%%%%%%%%%%%%%%%%%%%%%%%%%%%%%%%%%%%%%%%%%%%%%%%%

\end{document}